\documentclass[a4paper, 12pt]{article}

\usepackage[all]{xy}

\usepackage{amssymb, amscd, amsthm}

\theoremstyle{plain}
\newtheorem{theorem*}{Theorem}
\newtheorem{corollary*}[theorem*]{Corollary}
\newtheorem{theorem}{Theorem}[section]
\newtheorem{lemma}[theorem]{Lemma}
\newtheorem{proposition}[theorem]{Proposition}
\newtheorem{corollary}[theorem]{Corollary}

\theoremstyle{definition}

\newtheorem{remark}[theorem]{Remark}
\newtheorem{remarks}[theorem]{Remarks}

\newtheorem{examples}[theorem]{Examples}

\newcommand\bA{{\mathbb A}}

\newcommand\bG{{\mathbb G}}

\newcommand\bP{{\mathbb P}}

\newcommand\bV{{\mathbb V}}

\newcommand\cC{{\mathcal C}}

\newcommand\cL{{\mathcal L}}

\newcommand\cO{{\mathcal O}}

\newcommand\fm{\mathfrak{m}}

\newcommand\car{{\rm char}}

\newcommand\id{{\rm id}}

\newcommand\pr{{\rm pr}}

\newcommand\N{{\rm N}}
\newcommand\R{{\rm R}}

\newcommand\Alb{{\rm Alb}}

\newcommand\Frac{{\rm Frac}}
\newcommand\Gal{{\rm Gal}}

\newcommand\Ker{{\rm Ker}}
\newcommand\Lie{{\rm Lie}}

\newcommand\NS{{\rm NS}}
\newcommand\Pic{{\rm Pic}}

\newcommand\Spec{{\rm Spec}}

\title{Algebraic group actions on normal varieties}

\author{Michel Brion}

\date{}

\begin{document}

\maketitle

\begin{abstract}
Let $G$ be a connected algebraic $k$-group acting on a normal
$k$-variety, where $k$ is a field. We show that $X$ is covered
by open $G$-stable quasi-projective subvarieties; moreover, 
any such subvariety admits an equivariant embedding into the
projectivization of a $G$-linearized vector bundle on an abelian
variety, quotient of $G$. This generalizes a classical result
of Sumihiro for actions of smooth connected affine algebraic groups.
\end{abstract}


\section{Introduction and statement of the main results}
\label{sec:int}

Consider an algebraic $k$-group $G$ acting on a $k$-variety $X$, where 
$k$ is a field. If $X$ is normal and $G$ is smooth, connected and affine, 
then $X$ is covered by open $G$-stable quasi-projective subvarieties; 
moreover, any such variety admits a $G$-equivariant immersion in 
the projectivization of some finite-dimensional $G$-module. 
This fundamental result, due to Sumihiro (see 
\cite[Thm.~1, Lem.~8]{Sumihiro} and 
\cite[Thm.~2.5, Thm.~3.8]{Sumihiro-II}), has many applications.
For example, it yields that $X$ is covered by $G$-stable affine 
opens when $G$ is a split $k$-torus; this is the starting 
point of the classification of toric varieties (see \cite{CLS})
and more generally, of normal varieties with a torus action 
(see e.g. \cite{AHS, Langlois, LS}). 

Sumihiro's theorem does not extend directly to actions of arbitrary 
algebraic groups. For example, a non-trivial abelian variety $A$, 
acting on itself by translations, admits no equivariant embedding 
in the projectivization of a finite-dimensional $A$-module, since $A$ 
acts trivially on every such module. Also, an example of Hironaka (see 
\cite{Hironaka}) yields a smooth complete threefold equipped with an
involution $\sigma$ and which is not covered by $\sigma$-stable
quasi-projective opens. Yet a generalization of Sumihiro's 
theorem was obtained in \cite{Brion10} for actions of smooth connected 
algebraic groups over an algebraically closed field. The purpose 
of this article is to extend this result to an arbitrary field. 

More specifically, for any connected algebraic group $G$, we will prove:

\begin{theorem*}\label{thm:cover}
Every normal $G$-variety is covered by $G$-stable quasi-projective opens.
\end{theorem*}

\begin{theorem*}\label{thm:model} 
Every normal quasi-projective $G$-variety admits a $G$-equivariant 
immersion in the projectivization of a $G$-linearized vector bundle 
on an abelian variety, quotient of $G$ by a normal subgroup scheme.
\end{theorem*}

See the beginning of \S \ref{subsec:fp} for unexplained 
notation and conventions. Theorem \ref{thm:cover} is proved in
\S \ref{subsec:cover}, and Theorem \ref{thm:model} in \S 
\ref{subsec:model}. 

Theorem \ref{thm:cover} also follows from a result of Olivier Benoist
asserting that every normal variety contains finitely many maximal
quasi-projective open subvarieties (see \cite[Thm.~9]{Benoist}),
as pointed out by W\l odarczyk (see \cite[Thm.~D]{Wlodarczyk})
who had obtained an earlier version of the above result under more 
restrictive assumptions.

When $G$ is affine, any abelian variety quotient of $G$ is trivial, 
and hence the $G$-linearized vector bundles occuring in Theorem 
\ref{thm:model} are just the finite-dimensional $G$-modules. Thus,
Theorems \ref{thm:cover} and \ref{thm:model} give back Sumihiro's results.

Also, for a smooth connected algebraic group $G$ over a perfect field $k$, 
there exists a unique exact sequence of algebraic groups
$1 \to H \to G \to A \to 1$,
where $H$ is smooth, connected and affine, and $A$ is an abelian variety
(Chevalley's structure theorem, see \cite{Conrad,Milne} for modern
proofs). Then the $G$-linearized vector bundles occuring in 
Theorem \ref{thm:model} are exactly the homogeneous vector bundles 
$G \times^{H'} V$ on $G/H'$, where $H' \triangleleft G$ is a normal 
subgroup scheme containing $H$ (so that $G/H'$ is an abelian variety, 
quotient of $G/H = A$) and $V$ is a finite-dimensional $H'$-module.

The vector bundles on an abelian variety $A$ which are 
$G$-linearizable for some algebraic group $G$ with quotient $A$
are exactly the homogeneous, or translation-invariant, bundles; 
over an algebraically closed field, they have been classified by 
Miyanishi (see \cite[Thm.~2.3]{Miyanishi}) and Mukai (see
\cite[Thm.~4.17]{Mukai}).

We now present some applications of Theorems \ref{thm:cover} and 
\ref{thm:model}. First, as a straightforward consequence of Theorem
\ref{thm:model}, \emph{every normal quasi-projective $G$-variety $X$ 
admits an equivariant completion}, i.e., $X$ is isomorphic 
to a $G$-stable open of some complete $G$-variety. When $G$
is smooth and linear, this holds for any normal $G$-variety $X$
(not necessarily quasi-projective), by a result of Sumihiro
again; see \cite[Thm.~3]{Sumihiro}, \cite[Thm.~4.13]{Sumihiro-II}.
We do not know whether this result extends to an arbitrary
algebraic group $G$.

Another direct consequence of Theorems \ref{thm:cover} and
\ref{thm:model} refines a classical result of Weil:

\begin{corollary*}\label{cor:bir}
Let $X$ be a geometrically integral variety equipped with
a birational action of a smooth connected algebraic group $G$.
Then $X$ is $G$-birationally isomorphic to a normal projective
$G$-variety.
\end{corollary*} 

(Again, see the beginning of Subsection \ref{subsec:fp} for unexplained
notation and conventions). More specifically, Weil showed that
$X$ is $G$-birationally isomorphic to a normal $G$-variety $X'$
(see \cite[Thm.~p.~355]{Weil}). That $X'$ may be chosen projective 
follows by combining Theorems \ref{thm:cover} and \ref{thm:model}.
If $\car(k) = 0$, then we may assume in addition that $X'$ is
\emph{smooth} by using equivariant resolution of singularities
(see \cite[Thm.~3.36, Prop.~3.9.1]{Kollar}). The existence of 
such smooth projective ``models'' fails over any 
imperfect field (see e.g. \cite[Rem.~5.2.3]{Brion17}); one may 
ask whether \emph{regular} projective models exist 
in that setting.

Finally, like in \cite{Brion10}, we may reformulate Theorem \ref{thm:model} 
in terms of the Albanese variety, if $X$ is geometrically integral:
then $X$ admits a universal morphism to a torsor $\Alb^1(X)$ under 
an abelian variety $\Alb^0(X)$ (this is proved in \cite[Thm.~5]{Serre} 
when $k$ is algebraically closed, and extended to an arbitrary 
field $k$ in \cite[App.~A]{Wittenberg}).

\begin{corollary*}\label{cor:alb}
Let $X$ be a geometrically integral variety equipped with
an action $\alpha$ of a smooth connected algebraic group $G$. 
Then $\alpha$ induces an action $\Alb^1(\alpha)$ of $\Alb^0(G)$ 
on $\Alb^1(X)$. If $X$ is normal and quasi-projective, and $\alpha$ 
is almost faithful, then $\Alb^1(\alpha)$ is almost faithful as well.
\end{corollary*}

This result is proved in \S \ref{subsec:cor}. For a faithful action 
$\alpha$, it may happen that $\Alb^1(\alpha)$ is not faithful, see 
Remark \ref{rem:pr3}.

The proofs of Theorems \ref{thm:cover} and \ref{thm:model} follow the 
same lines as those of the corresponding results of \cite{Brion10}, 
which are based in turn on the classical proof of the projectivity 
of abelian varieties, and its generalization by Raynaud to the 
quasi-projectivity of torsors (see \cite{Raynaud} and also 
\cite[Chap.~6]{BLR}). But many arguments of \cite{Brion10} require 
substantial modifications, since the irreducibility and normality
assumptions on $X$ are not invariant under field extensions. 

Also, note that non-smooth subgroup schemes occur inevitably in 
Theorem \ref{thm:model} when $\car(k) = p > 0$: 
for example, the above subgroup schemes $H' \subset G$ obtained 
as pull-backs of $p$-torsion subgroup schemes of $A$ (see 
Remark \ref{rem:fin} (ii) for additional examples). Thus, we devote 
a large part of this article to developing techniques of algebraic 
transformation groups over an arbitrary field. 

Along the way, we obtain a generalization of Sumihiro's theorem 
in another direction: any normal quasi-projective variety equipped 
with an action of an affine algebraic group $G$ - not necessarily
smooth or connected - admits an equivariant immersion in the 
projectivization of a finite-dimensional $G$-module (see Corollary
\ref{cor:GH}).

This article is the third in a series devoted to the structure 
and actions of algebraic groups over an arbitrary field (see
\cite{Brion15, Brion17}). It replaces part of the unsubmitted
preprint \cite{Brion14}; the remaining part, dealing with
semi-normal varieties, will be developed elsewhere.

\section{Preliminaries}
\label{sec:prel}

\subsection{Functorial properties of algebraic group actions}
\label{subsec:fp}

Throughout this article, we fix a field $k$ with algebraic 
closure $\bar{k}$ and separable closure $k_s \subset \bar{k}$.
Unless otherwise specified, we consider separated schemes over 
$k$; morphisms and products of schemes are understood 
to be over $k$. The structure map of such a scheme $X$ is denoted 
by $q = q_X : X \to \Spec(k)$, and the scheme obtained by base change
under a field extension $k'/k$ is denoted by $X \otimes_k k'$, or just 
by $X_{k'}$ if this yields no confusion. A \emph{variety} is an integral 
scheme of finite type over $k$.

Recall that a \emph{group scheme} is a scheme $G$ equipped
with morphisms 
\[ \mu = \mu_G : G \times G \longrightarrow G, \quad  
\iota = \iota_G: G \longrightarrow G \] and
with a $k$-rational point $e = e_G \in G(k)$ such that for
any scheme $S$, the set of $S$-points $G(S)$ is a group with
multiplication map $\mu(S)$, inverse map $\iota(S)$ and 
neutral element $e \circ q_S \in G(S)$. 
This is equivalent to the commutativity of the following diagrams:
\[
\xymatrix{
G \times G \times G \ar[r]^-{\mu \times \id}\ar[d]_{\id \times \mu} 
& G \times G \ar[d]^{\mu} \\
G \times G \ar[r]^-{\mu} & G \\}
\]
(i.e., $\mu$ is associative),
\[
\xymatrix{
G \ar[r]^-{e \circ q \times \id} \ar[dr]_{\id} & 
G \times G \ar[d]^{\mu} & 
\ar[l]_-{\id \times e \circ q} \ar[dl]^{\id} G \\
& G \\
}
\]
(i.e., $e$ is the neutral element), and
\[
\xymatrix{
G \ar[r]^-{\id \times \iota} \ar[dr]_{e \circ q} & 
G \times G \ar[d]^{\mu} & 
\ar[l]_-{\iota \times \id} \ar[dl]^{e \circ q} G \\
& G \\
}
\]
(i.e., $\iota$ is the inverse map). We denote for simplicity
$\mu(g,h)$ by $g h$, and $\iota(g)$ by $g^{-1}$.
An \emph{algebraic group} is a group scheme of finite type
over $k$.

Given a group scheme $G$, a \emph{$G$-scheme} is a scheme $X$
equipped with a \emph{$G$-action}, i.e., a morphism
$\alpha : G \times X \to X$ such that for any scheme $S$, the map
$\alpha(S)$ defines an action of the group $G(S)$ on the set $X(S)$.
Equivalently, the following diagrams are commutative:
\[
\xymatrix{
G \times G \times X \ar[r]^-{\mu \times \id_X}\ar[d]_{\id_G \times \alpha} 
& G \times X \ar[d]^{\alpha} \\
G \times X \ar[r]^-{\alpha} & X \\}
\]
(i.e., $\alpha$ is ``associative''), and 
\[
\xymatrixcolsep{3pc}\xymatrix{
X  \ar[r]^-{e \circ q \times \id_X}\ar[dr]_{\id_X}  & 
G \times X \ar[d]^{\alpha} \\
&  X \\
}
\]
(i.e., the neutral element acts via the identity). We denote for
simplicity $\alpha(g,x)$ by $g \cdot x$. 

The \emph{kernel} of $\alpha$ is the group functor that assigns 
to any scheme $S$, the subgroup of $G(S)$ consisting of those 
$g \in G(S)$ that act trivially on the $S$-scheme $X \times S$ 
(i.e., $g$ acts trivially on the set $X(S')$ for any $S$-scheme $S'$). 
By \cite[II.1.3.6]{DG}, this group functor is represented by 
a closed normal subgroup scheme $\Ker(\alpha) \triangleleft G$.
Also, note that the formation of $\Ker(\alpha)$ commutes with
base change by field extensions. We say that $\alpha$ is 
\emph{faithful} (resp.~\emph{almost faithful}) if its kernel is trivial
(resp.~finite); then $\alpha_{k'}$ is faithful (resp.~almost faithful)
for any field extension $k'/k$.

A \emph{morphism of group schemes} is a morphism $f: G \to H$, 
where of course $G$, $H$ are group schemes, and 
$f(S) : G(S) \to H(S)$ is a group homomorphism for any scheme $S$.
Equivalently, the diagram
\[
\xymatrix{
G \times G \ar[r]^-{\mu_G}\ar[d]_{f \times f} 
& G \ar[d]^{f} \\
H \times H \ar[r]^-{\mu_H} & H \\}
\]
commutes.

Consider a morphism of group schemes $f: G \to H$, 
a scheme $X$ equipped with a $G$-action $\alpha$, a scheme $Y$ 
equipped with an $H$-action $\beta$ and a morphism (of schemes)
$\varphi : X \to Y$. We say that 
\emph{$\varphi$ is equivariant relative to $f$} if we have
$\varphi(g \cdot x) = f(g) \cdot \varphi(x)$ for any scheme
$S$ and any $g \in G(S)$, $x \in X(S)$. This amounts to the
commutativity of the diagram
\[
\xymatrix{
G \times X \ar[r]^-{\alpha}\ar[d]_{f \times \varphi} 
& X \ar[d]^{\varphi} \\
H \times Y \ar[r]^-{\beta} & Y. \\}
\]

We now recall analogues of some of these notions in birational geometry.
A \emph{birational action} of a smooth connected algebraic group 
$G$ on a variety $X$ is a rational map 
\[ \alpha : G \times X \dasharrow X \]
which satisfies the ``associativity'' condition on some open dense
subvariety of $G \times G \times X$, and such that the rational map
\[ G \times X \dasharrow G \times X, \quad 
(g,x) \longmapsto (g, \alpha(g,x)) \]
is birational as well. We say that two varieties $X$, $Y$ equipped
with birational actions $\alpha$, $\beta$ of $G$ are 
$G$-\emph{birationally isomorphic} if there exists a birational
map $\varphi : X \dasharrow Y$ which satisfies the equivariance
condition on some open dense subvariety of $G \times X$.

Returning to the setting of actions of group schemes, 
recall that a vector bundle $\pi : E \to X$ on a 
$G$-scheme $X$ is said to be $G$-\emph{linearized} if $E$
is equipped with an action of $G \times \bG_m$ such that
$\pi$ is equivariant relative to the first projection 
$\pr_G : G \times \bG_m \to G$, and $\bG_m$ acts on $E$ by 
multiplication on fibers. For a line bundle $L$, this is equivalent 
to the corresponding invertible sheaf $\cL$ (consisting of local 
sections of the dual line bundle) being $G$-linearized in the sense 
of \cite[Def.~1.6]{MFK}.

Next, we present some functorial properties of these notions, which 
follow readily from their definitions via commutative diagrams. 
Denote by $Sch_k$ the category of schemes over $k$. Let $\cC$ be 
a full subcategory of $Sch_k$ such that $\Spec(k) \in \cC$ and 
$X  \times Y \in \cC$ for all $X,Y \in \cC$. Let 
$F : \cC \to Sch_{k'}$ be a (covariant) functor, 
where $k'$ is a field. Following \cite[II.1.1.5]{DG}, we say that 
\emph{$F$ commutes with finite products} if 
$F(\Spec(k)) = \Spec(k')$ and the map
\[ F(\pr_X) \times F(\pr_Y) : 
F(X \times Y) \longrightarrow F(X) \times F(Y) \]
is an isomorphism for all $X,Y \in \cC$, where 
$\pr_X: X \times Y \to X$, $\pr_Y : X \times Y \to Y$ denote
the projections. 

Under these assumptions, $F(G)$ is equipped with a $k'$-group
scheme structure for any $k$-group scheme $G \in \cC$. Moreover, 
for any $G$-scheme $X \in \cC$, we obtain an $F(G)$-scheme
structure on $F(X)$. If $f : G \to H$ is a morphism of
$k$-group schemes and $G,H \in \cC$, then the morphism
$F(f): F(G) \to F(H)$ is a morphism of $k'$-group schemes.
If in addition $Y \in \cC$ is an $H$-scheme and 
$\varphi : X \to Y$ an equivariant morphism relative to $f$, 
then the morphism $F(\varphi) : F(X) \to F(Y)$ is equivariant 
relative to $F(f)$. 

Also, if $F_1 : \cC \to Sch_{k_1}$, $F_2 : \cC \to Sch_{k_2}$
are two functors commuting with finite products, and 
$T : F_1 \to F_2$ is a morphism of functors, 
then $T$ induces morphisms of group schemes 
$T(G) : F_1(G) \to F_2(G)$, and equivariant morphisms 
$T(X) : F_1(X) \to F_2(X)$ relative to $T(G)$, for all $G,X$ as above.

Consider again a functor $F : \cC \to Sch_{k'}$ commuting
with finite products. We say that 
\emph{$F$ preserves line bundles} if for any line bundle
$\pi : L \to X$, where $X \in \cC$, we have that $L \in \cC$
and $F(\pi) : F(L) \to F(X)$ is a line bundle; in addition,
we assume that $\bG_{m,k} \in \cC$ and $F(\bG_{m,k}) \cong \bG_{m,k'}$ 
compatibly with the action of $\bG_{m,k}$ on $L$ by multiplication
on fibers, and the induced action of $F(\bG_{m,k})$ on $F(L)$.
Under these assumptions, for any $G$-scheme $X \in \cC$ and
any $G$-linearized line bundle $L$ on $X$, the line bundle
$F(L)$ on $F(X)$ is equipped with an $F(G)$-linearization.

\begin{examples}\label{ex:fp}
(i) Let $h : k \to k'$ be a homomorphism of fields. Then the
base change functor 
\[ F : Sch_k \longrightarrow Sch_{k'}, 
\quad X \longmapsto X \otimes_h k' := X \times_{\Spec(k)} \Spec(k') \] 
commutes with finite products and preserves line bundles. Also,
assigning to a $k$-scheme $X$ the projection 
\[ \pr_X : X \otimes_h k' \longrightarrow X \]
yields a morphism of functors from $F$ to the identity of 
$Sch_k$. As a consequence, $G \otimes_h k'$ is a 
$k'$-group scheme for any $k$-group scheme $G$, and $\pr_G$ is a 
morphism of group schemes. Moreover, for any $G$-scheme $X$, 
the scheme $X \otimes_h k'$ comes with an action of
$G \otimes_h k'$ such that $\pr_X$ is equivariant; also,
every $G$-linearized line bundle $L$ on $X$ yields a
$G \otimes_h k'$-linearized line bundle $L \otimes_h k'$ on
$X \otimes_h k'$. This applies for instance to the Frobenius 
twist $X \mapsto X^{(p)}$ in characteristic $p > 0$ 
(see Subsection \ref{subsec:ifm} for details).

\medskip

\noindent
(ii) Let $k'/k$ be a finite extension of fields, and $X'$ 
a quasi-projective scheme over $k'$. Then the Weil restriction 
$\R_{k'/k}(X')$ is a quasi-projective scheme over $k$ (see 
\cite[7.6]{BLR} and \cite[A.5]{CGP} for details on Weil restriction). 
The assignment $X' \mapsto \R_{k'/k}(X')$ extends to a functor 
\[ \R_{k'/k} : Sch^{qp}_{k'} \longrightarrow Sch_k^{qp}, \] 
where $Sch_k^{qp}$ denotes the full subcategory of $Sch_k$ 
with objects being the quasi-projective schemes.
By \cite[A.5.2]{CGP}, $\R_{k'/k}$ commutes with finite products, 
and hence so does the functor
\[ F : Sch_k^{qp} \longrightarrow Sch_k^{qp}, \quad
X \longmapsto \R_{k'/k}(X_{k'}). \]
Since every algebraic group $G$ is quasi-projective (see 
e.g.~\cite[A.3.5]{CGP}), we see that $\R_{k'/k}(G_{k'})$ is equipped 
with a structure of $k$-group scheme. Moreover, for any 
quasi-projective $G$-scheme $X$, we obtain an 
$\R_{k'/k}(G_{k'})$-scheme structure on $\R_{k'/k}(X_{k'})$.
The adjunction morphism
\[ j_X : X \longrightarrow \R_{k'/k}(X_{k'}) = F(X) \] 
is a closed immersion by \cite[A.5.7]{CGP}, and extends to 
a morphism of functors from the identity of $Sch_k^{qp}$ to
the endofunctor $F$. As a consequence, for any quasi-projective
$G$-scheme $X$, the morphism $j_X$ is equivariant relative
to $j_G$.

Note that $F$ does not preserve line bundles (unless $k'= k$), 
since the algebraic $k'$-group $\R_{k'/k}(\bG_{m,k'})$ has 
dimension $[k':k]$. 

\medskip

\noindent
(iii) Let $X$ be a scheme, locally of finite type over $k$.
Then there exists an \'etale scheme $\pi_0(X)$ and a morphism
\[ \gamma = \gamma_X : X \longrightarrow \pi_0(X), \]
such that every morphism $f :X \to Y$, where $Y$ is \'etale,
factors uniquely through $\gamma$. Moreover, $\gamma$ is faithfully 
flat, and its fibers are exactly the connected components of $X$. 
The formation of $\gamma$ commutes with field extensions
and finite products (see \cite[I.4.6]{DG} for these results).
In particular, $X$ is connected if and only if $\pi_0(X) = \Spec(K)$
for some finite separable field extension $K/k$. Also, $X$ is 
geometrically connected if and only if $\pi_0(X) = \Spec(k)$. 

As a well-known consequence, for any group scheme $G$, 
locally of finite type, we obtain a group scheme structure on 
the \'etale scheme $\pi_0(G)$ such that $\gamma_G$ is a morphism 
of group schemes; its kernel is the neutral component $G^0$. 
Moreover, any action of $G$ on a scheme of finite type $X$ 
yields an action of $\pi_0(G)$ on $\pi_0(X)$ such that $\gamma_X$ 
is equivariant relative to $\gamma_G$. In particular, 
every connected component of $X$ is stable under $G^0$.

\medskip

\noindent
(iv) Consider a connected scheme of finite type $X$,
and the morphism $\gamma_X : X \to \Spec(K)$ as in (iii).
Note that the degree $[K:k]$ is the number of geometrically connected
components of $X$. Also, we may view $X$ as a $K$-scheme; then it
is geometrically connected.

Given a $k$-scheme $Y$, the map 
\[ \iota_{X,Y} := \id_X \times \pr_Y : X \times_K Y_K 
\longrightarrow X \times_k Y \]
is an isomorphism of $K$-schemes, where $X \times_k Y$ is viewed as
a $K$-scheme via $\gamma_X \circ \pr_X$. Indeed, considering open affine
coverings of $X$ and $Y$, this boils down to the assertion that the map 
\[ R \otimes_k S \longrightarrow R \otimes_K (S \otimes_k K),
\quad r \otimes s \longmapsto r \otimes (s \otimes 1) \]
is an isomorphism of $K$-algebras for any $K$-algebra $R$ and
any $k$-algebra $S$.

Also, note that the projection $\pr_X : X_K \to X$
has a canonical section, namely, the adjunction map 
$\sigma_X : X \to X_K$. Indeed, considering an open affine
covering of $X$, this reduces to the fact that the inclusion map 
\[ R \longrightarrow R \otimes_k K, 
\quad r \longmapsto r \otimes 1 \]
has a retraction given by $r \otimes z \mapsto z r$.
Thus, $\sigma_X$ identifies the $K$-scheme $X$ with a connected 
component of $X_K$. 

For any $k$-scheme $Y$, the above map $\iota_{X,Y}$ is compatible 
with $\sigma_X$ in the sense that the diagram
\[
\xymatrixcolsep{4pc}\xymatrix{
X \times_K Y_K \ar[r]^-{\iota_{X,Y}} \ar[d]_{\sigma_X \times \id_Y} & 
X \times_k Y \ar[d]^{\sigma_{X \times_k Y}}\\
X_K \times_K Y_K  \ar[r]^-{\pr_X \times \id_{Y_K}}
& (X \times_k Y)_K \\
}
\]
commutes, with the horizontal maps being isomorphisms. Indeed, 
this follows from the identity 
$z r \otimes s \otimes 1 = r \otimes s \otimes z$
in $R \otimes_K (S \otimes_k K)$ for any $R$ and $S$ as above, and
any $z \in K$, $r \in R$, $s \in S$.

Given a morphism of $k$-schemes $f : X' \to X$, we may also view 
$X'$ as a $K$-scheme via the composition $X' \to X \to \Spec(K)$.
Then the diagram 
\[
\xymatrix{
X' \ar[r]^f \ar[d]_{\sigma_{X'}} & X \ar[d]^{\sigma_X} \\
X'_K \ar[r]^{f_K} & X_K \\
}
\]
commutes, as may be checked by a similar argument.

In particular, if $X$ is equipped with an action of an algebraic 
$k$-group $G$, then $G_K$ acts on the $K$-scheme $X$ through the 
morphism $\pr_G : G_K \to G$; moreover, $X$ is stable under the
induced action of $G_K$ on $X_K$, since the diagram
\[
\xymatrixcolsep{4pc}\xymatrix{
G_K \times_K X \ar[r]^-{\iota_{X,G}} \ar[d]_{\id_{G_K} \times \sigma_X} & 
G \times_k X \ar[r]^-{\alpha} \ar[d]_{\sigma_{G \times_k X}} 
& X \ar[d]^{\sigma_X} \\
G_K \times_K X_K  \ar[r]^-{\id_{G_K} \times \pr_X}
& (G \times_k X)_K \ar[r]^-{\alpha_K} & X_K \\
}
\]
commutes.

When $X$ is a normal $k$-variety, the above field $K$ is
the separable algebraic closure of $k$ in the function field 
$k(X)$. (Indeed, $K$ is a subfield of $k(X)$ as $\gamma_X$ is 
faithfully flat; hence $K \subset L$, where $L$ denotes the
separable algebraic closure of $k$ in $k(X)$. On the other hand,
$L \subset \cO(X)$ as $L \subset k(X)$ is integral over $k$.
This yields a morphism $X \to \Spec(L)$, and hence a homomorphism
$L \to K$ in view of the universal property of $\gamma_X$.
Thus, $L = K$ for degree reasons). Since the $K$-scheme $X$
is geometrically connected, we see that $X \otimes_K K_s$ 
is a normal $K_s$-variety. In particular, $X$ is geometrically 
irreducible as a $K$-scheme. 
\end{examples}

\subsection{Norm and Weil restriction}
\label{subsec:nwr}

Let $k'/k$ be a finite extension of fields, and $X$ a $k$-scheme.
Then the projection 
\[ \pr_X : X_{k'} \longrightarrow  X \] 
is finite and the sheaf of $\cO_X$-modules $(\pr_X)_*(\cO_{X_{k'}})$ is 
locally free of rank $[k':k] =: n$. Thus, we may assign to any line bundle 
\[ \pi : L' \longrightarrow X_{k'}, \] 
its \emph{norm} $\N(L')$; this is a line bundle on $X$, unique up to 
unique isomorphism (see \cite[II.6.5.5]{EGA}). Assuming that $X$ is 
quasi-projective, we now obtain an interpretation of $\N(L')$ in terms 
of Weil restriction:

\begin{lemma}\label{lem:nwr}
Keep the above notation, and the notation of Example \ref{ex:fp} (ii).

\begin{enumerate}

\item[{\rm (i)}]
The map $\R_{k'/k}(\pi) : \R_{k'/k}(L') \to \R_{k'/k}(X_{k'})$ 
is a vector bundle of rank $n$. 

\item[{\rm (ii)}]
We have an isomorphism of line bundles on $X$
\[ \N(L') \cong  j_X^*  \det \R_{k'/k}(L'). \]

\item[{\rm (iii)}]
If $X$ is equipped with an action of an algebraic group $G$
and $L'$ is $G_{k'}$-linearized, then $\N(L')$ is $G$-linearized.

\end{enumerate}

\end{lemma}

\begin{proof}
(i) Let $E := \R_{k'/k}(L')$ and $X' := \R_{k'/k}(X_{k'})$. 
Consider the $\bG_{m,k'}$-torsor 
\[ \pi^{\times} : L'^{\times} \longrightarrow X_{k'} \] 
associated with the line bundle $L'$. Recall that
$L' \cong (L'^{\times} \times \bA^1_{k'})/\bG_{m,k'}$, where $\bG_{m,k'}$
acts simultaneously on $L'^{\times}$ and on $\bA^1_{k'}$ 
by multiplication. Using \cite[A.5.2, A.5.4]{CGP}, it follows that
\[
E \cong 
(\R_{k'/k}(L'^{\times}) \times \R_{k'/k}(\bA^1_{k'}))/\R_{k'/k}(\bG_{m,k'}).
\]
This is the fiber bundle on $X'$ associated with the 
$\R_{k'/k}(\bG_{m,k'})$-torsor $\R_{k'/k}(L'^{\times}) \to X'$ and
the $\R_{k'/k}(\bG_{m,k'})$-scheme $\R_{k'/k}(\bA^1_{k'})$. Moreover,
$\R_{k'/k}(\bA^1_{k'})$ is the affine space $\bV(k')$ associated with 
the $k$-vector space $k'$ on which $\R_{k'/k}(\bG_{m,k'})$ acts
linearly, and $\bG_{m,k}$ (viewed as a subgroup scheme of 
$\R_{k'/k}(\bG_{m,k'})$ via the adjunction map) acts by scalar 
multiplication. Indeed, for any $k$-algebra $A$, we have
$\R_{k'/k}(\bA^1_{k'})(A)= A \otimes_k k' = A_{k'}$ on which 
$\R_{k'/k}(\bG_{m,k'})(A) = A_{k'}^*$ and its subgroup 
$\bG_{m,k}(A) = A^*$ act by multiplication.
 
(ii) The determinant of $E$ is the line bundle associated with the above 
$\R_{k'/k}(\bG_{m,k'})$-torsor and the 
$\R_{k'/k}(\bG_{m,k'})$-module $\bigwedge^n(k')$
(the top exterior power of the $k$-vector space $k'$). 
To describe the pull-back of this line bundle under 
$j_X: X \to X'$, choose a Zariski open covering $(U_i)_{i \in I}$ 
of $X$ such that the $(U_i)_{k'}$ cover $X_{k'}$ and the pull-back 
of $L'$ to each $(U_i)_{k'}$ is trivial (such a covering exists by
\cite[IV.21.8.1]{EGA}). Also, choose trivializations 
\[ \eta_i : L'_{(U_i)_{k'}} \stackrel{\cong}{\longrightarrow}
(U_i)_{k'} \times_{k'} \bA^1_{k'}. \]
This yields trivializations
\[ \R_{k'/k}(\eta_i) : E_{U'_i} \stackrel{\cong}{\longrightarrow} 
U'_i \times_k \bV(k'), \]
where $U'_i := \R_{k'/k}((U_i)_{k'})$. Note that the $U'_i$ do not 
necessarily cover $X'$, but the $j_X^{-1}(U'_i) = U_i$ do cover $X$. 
Thus, $j_X^*(E)$ is equipped with trivializations
\[  j_X^*(E)_{U_i} \stackrel{\cong}{\longrightarrow} 
U_i \times_k \bV(k'). \]
Consider the $1$-cocycle
$(\omega_{ij} := (\eta_i \eta_j^{-1})_{(U_i)_{k'} \cap (U_j)_{k'}})_{i,j}$
with values in $\bG_{m,k'}$. Then the line bundle 
$j_X^*(\det(E)) = \det(j_X^*(E))$ is defined by the $1$-cocycle 
$(\det(\omega_{ij}))_{i,j}$ with values in $\bG_{m,k}$, 
where $\det(\omega_{ij})$ denotes the determinant of 
the multiplication by $\omega_{ij}$ in the $\cO(U_i \cap U_j)$-algebra
$\cO(U_i \cap U_j) \otimes_k k'$. 
It follows that $j_X^*(\det(E)) \cong \N(L')$ in view of the definition 
of the norm (see \cite[II.6.4, II.6.5]{EGA}). 

(iii) By Example \ref{ex:fp} (ii), $\R_{k'/k}(X_{k'})$
is equipped with an action of $\R_{k'/k}(G_{k'})$; moreover, $j_X$ 
is equivariant relative to $j_G : G \to \R_{k'/k}(G_{k'})$. Also, 
the action of $G_{k'} \times \bG_{m,k'}$ on $L'$ yields an action
of $\R_{k'/k}(G) \times \R_{k'/k}(\bG_{m,k'})$ on $E$ such that
$\bG_{m,k} \subset \R_{k'/k}(\bG_{m,k'})$ acts by scalar 
multiplication on fibers. Thus, the vector bundle $E$ is equipped 
with a linearization relative to $\R_{k'/k}(G_{k'})$, which induces 
a linearization of its determinant. This yields the assertion in view of (ii).
\end{proof}

\subsection{Iterated Frobenius morphisms}
\label{subsec:ifm}

In this subsection, we assume that $\car(k) = p > 0$. 
Then every $k$-scheme $X$ is equipped with the 
\emph{absolute Frobenius endomorphism} $F_X$: it induces 
the identity on the underlying topological space, 
and the homomorphism of sheaves of algebras
$F_X^{\#} : \cO_X \to (F_X)_*(\cO_X) = \cO_X$
is the $p$th power map, $f \mapsto f^p$. Note that $F_X$
is not necessarily a morphism of $k$-schemes, as the
structure map $q_X : X \to \Spec(k)$ lies in a commutative
diagram
\[
\xymatrix{
X \ar[r]^-{F_X} \ar[d]_{q_X} & X \ar[d]^{q_X} \\ 
\Spec(k) \ar[r]^-{F_k} & \Spec(k),  \\}
\]
where $F_k := F_{\Spec(k)}$. We may form the commutative diagram
\[
\xymatrix{X \ar[d]_{F_{X/k}} \ar[dr]^{F_X} \\
X^{(p)} \ar[r]^-{\pr_X} \ar[d]_{q_{X^{(p)}}} & X \ar[d]^{q_X} \\ 
\Spec(k) \ar[r]^-{F_k} & \Spec(k),  \\}
\]
where the square is cartesian and $F_{X/k} \circ q_{X^{(p)}} = q_X$.
In particular, $F_{X/k} : X \to X^{(p)}$
is a morphism of $k$-schemes: the \emph{relative Frobenius morphism}.
The underlying topological space of $X^{(p)}$ may be identified with
that of $X$; then $\cO_{X^{(p)}} = \cO_X \otimes_{F_k} k$ and
the morphism $F_{X/k}$ induces the identity on topological spaces,
while $F_{X/k}^{\#} : \cO_X \otimes_{F_k} k \to \cO_X$ is given by
$f \otimes z \mapsto z f^p$.

The assignment $X \mapsto X^{(p)}$ extends to a covariant endofunctor
of the category of schemes over $k$, 
which commutes with products and field extensions; moreover, 
the assignment $X \mapsto F_{X/k}$ extends to a morphism of functors
(see e.g. \cite[VIIA.4.1]{SGA3}). In view of Subsection \ref{subsec:fp}, 
it follows that for any $k$-group scheme $G$, there is a canonical 
$k$-group scheme structure on $G^{(p)}$ such that 
$F_{G/k} : G \to G^{(p)}$ is a morphism of group schemes.
Its kernel is called the \emph{Frobenius kernel} of $G$;
we denote it by $G_1$. Moreover, for any $G$-scheme $X$, 
there is a canonical $G^{(p)}$-scheme structure on $X^{(p)}$ 
such that $F_{X/k}$ is equivariant relative to $F_G$.

By \cite[XV.1.1.2]{SGA5}, the morphism $F_{X/k}$ is integral, surjective 
and radicial; equivalently, $F_{X/k}$ is a universal homeomorphism
(recall that a morphism of schemes is \emph{radicial} 
if it is injective and induces purely inseparable extensions of residue 
fields). Thus, $F_{X/k}$ is finite if $X$ is of finite type over $k$; then 
$X^{(p)}$ is of finite type over $k$ as well, since it is obtained from
$X$ by the base change $F_k : \Spec(k) \to \Spec(k)$. In particular,
for any algebraic group $G$, the Frobenius kernel $G_1$ is finite 
and radicial over $\Spec(k)$. Equivalently, $G_1$ is an
\emph{infinitesimal group scheme}. 

Next, let $\cL$ be an invertible sheaf on $X$, and 
$f : L \to X$ the corresponding line bundle. Then
$f^{(p)} : L^{(p)} \to X^{(p)}$ is a line bundle, and there is 
a canonical isomorphism 
\[ F_{X/k}^*(L^{(p)}) \cong L^{\otimes p} \]
(see \cite[XV.1.3]{SGA5}). If $X$ is a $G$-scheme and $L$ is 
$G$-linearized, then $L^{(p)}$ is $G^{(p)}$-linearized as well,
in view of Example \ref{ex:fp} (i). Also, note that 
\emph{$L$ is ample if and only if $L^{(p)}$ is ample}.
Indeed, $L^{(p)}$ is the base change of $L$ under $F_k$, and hence 
is ample if so is $L$ (see \cite[II.4.6.13]{EGA}).
Conversely, if $L^{(p)}$ is ample, then so is $F_{X/k}^*(L^{(p)})$
as $F_{X/k}$ is affine (see e.g. \cite[II.5.1.12]{EGA}); thus, 
$L$ is ample as well.

We now extend these observations to the 
\emph{iterated relative Frobenius morphism}
\[ F^n_{X/k} : X \longrightarrow X^{(p^n)}, \]
where $n$ is a positive integer. Recall from 
\cite[VIIA.4.1]{SGA3} that $F^n_{X/k}$ is defined inductively by
$F^1_{X/k} = F_{X/k}$, $X^{(p^n)} = (X^{(p^{n-1})})^{(p)}$ and
$F^n_{X/k}$ is the composition
\[ 
\xymatrixcolsep{5pc}\xymatrix{
X \ar[r]^{F_{X/k}} &
X^{(p)} \ar[r]^-{F_{X^{(p)}/k}} &
X^{(p^2)} \to \cdots \to 
X^{(p^{n-1})} \ar[r]^-{F_{X^{(p^{n-1})}/k}} &
X^{(p^n)}.
}
\]

This yields readily:

\begin{lemma}\label{lem:ifm}
Let $X$ be a scheme of finite type, $L$ a line bundle on $X$,
and $G$ an algebraic group.

\begin{enumerate}

\item[{\rm (i)}] The scheme $X^{(p^n)}$ is of finite type, and 
$F^n_{X/k}$ is finite, surjective and radicial.

\item[{\rm (ii)}] $F^n_{G/k} : G \to G^{(p^n)}$ is a morphism of
algebraic groups, and its kernel (the $n$th Frobenius kernel $G_n$)
is infinitesimal.

\item[{\rm (iii)}] $L^{(p^n)}$  is a line bundle on $X^{(p^n)}$, 
and we have a canonical isomorphism 
\[ (F^n_{X/k})^*(L^{(p^n)}) \cong L^{\otimes p^n}. \] 
Moreover, $L$ is ample if and only if $L^{(p^n)}$ is ample.

\item[{\rm (iv)}] If $X$ is a $G$-scheme, then $X^{(p^n)}$ is a
$G^{(p^n)}$-scheme and $F^n_{X/k}$ is equivariant relative to
$F^n_{G/k}$. If in addition $L$ is $G$-linearized, then
$L^{(p^n)}$ is $G^{(p^n)}$-linearized.

\end{enumerate}

\end{lemma}

\begin{remarks}\label{rem:ifm}
(i) If $X$ is the affine space $\bA^d_k$, then 
$X^{(p^n)} \cong \bA^d_k$ for all $n \geq 1$. More generally,
if $X \subset \bA^d_k$ is the zero subscheme of 
$f_1,\ldots,f_m \in k[x_1,\ldots,x_d]$, then 
$X^{(p^n)} \subset \bA^d_k$ is the zero subscheme of 
$f_1^{(p^n)}, \ldots, f_m^{(p^n)}$, where each $f_i^{(p^n)}$
is obtained from $f_i$ by raising all the coefficients to the
$p^n$th power.

\medskip

\noindent
(ii) Some natural properties of $X$ are not preserved
under Frobenius twist $X \mapsto X^{(p)}$. For example, assume 
that $k$ is imperfect and choose $a \in k \setminus k^p$, where 
$p := \car(k)$. Let $X := \Spec(K)$, where 
$K$ denotes the field $k(a^{1/p}) \cong k[x]/(x^p -a)$. Then
$X^{(p^n)} \cong \Spec(k[x]/(x^p -a^{p^n})) 
\cong \Spec(k[y]/(y^p))$ is non-reduced for all $n \geq 1$.

This can be partially remedied by replacing $X^{(p)}$
with the scheme-theoretic image of $F_{X/k}$; for example, 
one easily checks that this image is geometrically reduced
for $n \gg 0$. But given a normal variety $X$, it may happen 
that $F^n_{X/k}$ is an epimorphism and $X^{(p^n)}$ is non-normal 
for any $n \geq 1$. For example, take $k$ and $a$ as above
and let $X \subset \bA^2_k = \Spec(k[x,y])$ be the zero subscheme 
of $y^{\ell} - x^p + a$, where $\ell$ is a prime and $\ell \neq p$. 
Then $X$ is a regular curve: indeed, by the jacobian criterion, 
$X$ is smooth away from the closed point $P := (a^{1/p},0)$; also, 
the maximal ideal of $\cO_{X,P}$ is generated by the image of $y$, 
since the quotient ring
$k[x,y]/(y^{\ell} - x^p +a, y) \cong k[x]/(x^p - a)$ is a field.
Moreover, $X^{(p^n)} \subset \bA^2_k$ is the zero subscheme of
$y^{\ell} - x^p + a^{p^n}$, and hence is not regular at the point 
$(a^{p^{n-1}},0)$. Also, $F^n_{X/k}$ is an epimorphism as 
$X^{(p^n)}$ is integral.  
\end{remarks}

\subsection{Quotients by infinitesimal group schemes}
\label{subsec:qi}

Throughout this subsection, we still assume that $\car(k) = p > 0$.
Recall from \cite[VIIA.8.3]{SGA3} that for any algebraic group
$G$, there exists a positive integer $n_0$ such that the quotient 
group scheme $G/G_n$ is smooth for $n \ge n_0$. In particular,
for any infinitesimal group scheme $I$, there exists a positive
integer $n_0$ such that the $n$th Frobenius kernel $I_n$ is 
the whole $I$ for $n \geq n_0$. The smallest such integer is called 
the \emph{height} of $I$; we denote it by~$h(I)$.

\begin{lemma}\label{lem:quot}
Let $X$ be a scheme of finite type equipped with an action $\alpha$ 
of an infinitesimal group scheme $I$.

\begin{enumerate}

\item[{\rm (i)}] There exists a categorical quotient
\[ \varphi = \varphi_{X,I} : X \longrightarrow X/I, \]
where $X/I$ is a scheme of finite type and $\varphi$ is a finite, surjective, 
radicial morphism. 

\item[{\rm (ii)}] For any integer $n \geq h(I)$, the relative Frobenius
morphism $F^n_{X/k} : X \to X^{(p^n)}$ factors uniquely as
\[ X \stackrel{\varphi}{\longrightarrow} X/I 
\stackrel{\psi}{\longrightarrow} X^{(p^n)}. \]
Moreover, $\psi = \psi_{X,I}$ is finite, surjective and radicial as well.

\item[{\rm (iii)}] Let $n \geq h(I)$ and $L$ a line bundle on $X$.
Then $M := \psi^*(L^{(p^n)})$ is a line bundle on $X/I$, and 
$\varphi^*(M) \cong L^{\otimes p^n}$. Moreover, $L$ is ample
if and only if $M$ is ample.

\item[{\rm (iv)}] If $X$ is a normal variety, then so is $X/I$.

\end{enumerate}

\end{lemma}

\begin{proof}
(i) Observe that the morphism 
\[ \gamma := \id_X \times \alpha : I \times X \longrightarrow I \times X \]
is an $I$-automorphism and satisfies $\gamma \circ \pr_X = \alpha$
on $I \times X$. As $I$ is infinitesimal, the morphism $\pr_X$ is finite, 
locally free and bijective; thus, so is $\alpha$. In view of 
\cite[V.4.1]{SGA3}, it follows that the categorical quotient $\varphi$ 
exists and is integral and surjective. The remaining assertions 
will be proved in (ii) next.

(ii) By Lemma \ref{lem:ifm} (iv), $F^n_{X/k}$ is $I$-invariant for 
any $n \geq h(I)$. Since $\varphi$ is a categorical quotient, this yields
the existence and uniqueness of $\psi$. As $F^n_{X/k}$ is universally
injective, so is $\varphi$; equivalently, $\varphi$ is radicial. 
In view of (i), it follows that $\varphi$ is a universal homeomorphism. 
As $F^n_{X/k}$ is a universal homeomorphism as well, so is $\psi$.

Recall from Lemma \ref{lem:ifm} that $X^{(p^n)}$ is of finite type and 
$F^n_{X/k}$ is finite. As a consequence, $\varphi$ and $\psi$ are finite, 
and $X/I$ is of finite type.

(iii) The first assertion follows from Lemma  \ref{lem:ifm} (iii).
If $L$ is ample, then so is $L^{(p^n)}$ by that lemma; thus,
$M$ is ample as $\psi$ is affine. Conversely, if $M$ is ample,
then so is $L$ as $\varphi$ is affine.

(iv) Note that $X/I$ is irreducible, since $\varphi$ is a homeomorphism.
Using again the affineness of $\varphi$, we may thus assume that $X/I$,
and hence $X$, are affine. Then the assertion follows by a standard
argument of invariant theory. More specifically, let $X = \Spec(R)$,
then $R$ is an integral domain and $X/I = \Spec(R^I)$, where 
$R^I \subset R$ denotes the subalgebra of invariants, consisting of 
those $f \in R$ such that $\alpha^{\#}(f) = \pr_X^{\#}(f)$ in 
$\cO(I \times X)$. Thus, $R^I$ is a domain. We check that it is normal: 
if $f \in \Frac(R^I)$ is integral over $R^I$, then $f \in \Frac(R)$ is 
integral over $R$, and hence $f \in R$. To complete the proof,
it suffices to show that $f$ is invariant. But $f = \frac{f_1}{f_2}$
where $f_1, f_2 \in R^I$ and $f_2 \neq 0$; this yields
\[ 0 = \alpha^{\#}(f_1) - \pr_X^{\#}(f_1) 
= \alpha^{\#}(f f_2) - \pr_X^{\#}(f f_2) 
= (\alpha^{\#}(f) - \pr_X^{\#}(f)) \pr_X^{\#}(f_2) \]
in $\cO(I \times X) \cong \cO(I)  \otimes_k R$.
Via this isomorphism, $\pr_X^{\#}(f_2)$ is identified
with $1 \otimes f_2$, which is not a zero divisor in   
$\cO(I) \otimes_k R$ (since its image in 
$\cO(I) \otimes_k \Frac(R)$ is invertible). Thus,
$\alpha^{\#}(f) - \pr_X^{\#}(f) = 0$ as desired. 
\end{proof}

\begin{remark}\label{rem:qi}
With the notation of Lemma \ref{lem:quot}, we may identify 
the underlying topological space of $X/I$ with that of $X$,
since $\varphi$ is radicial. Then the structure sheaf $\cO_{X/I}$
is just the sheaf of invariants $\cO_X^I$. As a consequence,
$\varphi_{X,I}$ is an epimorphism.
\end{remark}

\begin{lemma}\label{lem:prod}
Let $X$ (resp. $Y$) be a scheme of finite type equipped with 
an action of an infinitesimal algebraic group $I$ (resp.~$J$). 
Then the morphism 
$\varphi_{X,I} \times \varphi_{Y,J} : X \times Y \to X/I \times Y/J$
factors uniquely through an isomorphism 
\[ (X \times Y)/(I \times J) \stackrel{\cong}{\longrightarrow} 
X/I \times Y/J. \]
\end{lemma}

\begin{proof}
Since $\varphi_{X,I} \times \varphi_{Y,J}$ is invariant under 
$I \times J$, it factors uniquely through a morphism
$f : (X \times Y)/(I \times J) \to X/I \times Y/J$. To show that
$f$ is an isomorphism, we may assume by Remark \ref{rem:qi}
that $X$ and $Y$ are affine. Let $R := \cO(X)$ and
$S := \cO(Y)$; then we are reduced to showing that the natural
map  
\[ f^{\#} : R^I \otimes S^J \longrightarrow (R \otimes S)^{I \times J} \]
is an isomorphism. Here and later in this proof, all tensor products
are taken over $k$.

Clearly, $f^{\#}$ is injective. To show the surjectivity, we consider first
the case where $J$ is trivial. Choose a basis $(s_a)_{a \in A}$ of the
$k$-vector space $S$. Let $f \in R \otimes S$ and write 
$f = \sum_{a \in A} r_a \otimes s_a$, where the $r_a \in R$ are unique.
Then $f \in (R \otimes S)^I$ if and only if 
$\alpha^*(f) = \pr_X^*(f)$ in $\cO(I \times X \times Y)$, i.e.,
\[ \sum_{a \in A} \alpha^*(r_a) \otimes s_a = 
\sum_{a \in A} \pr_X^*(r_a) \otimes s_a \] 
in $\cO(I) \otimes R \otimes S$. As the $s_a$ are linearly independent 
over $\cO(I) \otimes R$, this yields $\alpha^*(r_a) = \pr_X^*(r_a)$, 
i.e., $r_a \in R^I$, for all $a \in A$. In turn, this yields 
$(R \otimes S)^I = R^I \otimes S$.

In the general case, we use the equality
\[ (R \otimes S)^{I \times J} = (R \otimes S)^I \cap (R \otimes S)^J \]
of subspaces of $R \otimes S$. In view of the above step, this yields
\[ (R \otimes S)^{I \times J} = (R^I \otimes S) \cap (R \otimes S^J). \]
Choose decompositions of $k$-vector spaces
$R = R^I \oplus V$ and $S = S^J \oplus W$; then we obtain a
decomposition
\[ R \otimes S = (R^I \otimes S^J) \oplus (R^I \otimes W) 
\oplus (V \otimes S^J) \oplus (V \otimes W), \]
and hence the equality
\[ (R^I \otimes S) \cap (R \otimes S^J) = R^I \otimes S^J. \]
\end{proof}

\begin{lemma}\label{lem:equiv}
Let $G$ be an algebraic group, $X$ a $G$-scheme of finite type 
and $n$ a positive integer. Then there exists a unique action of $G/G_n$ 
on $X/G_n$ such that the morphism $\varphi_{X,G_n} : X \to X/G_n$ 
(resp.~$\psi_{X,G_n} : X/G_n \to X^{(p^n)}$) is equivariant relative to 
$\varphi_{G,G_n} : G \to G/G_n$ 
(resp.~$\psi_{G,G_n} : G/G_n \to G^{(p^n)}$). 
\end{lemma}

\begin{proof}
Denote as usual by $\alpha : G \times X \to X$ the action 
and write for simplicity
$\varphi_X := \varphi_{X,G_n}$ and $\varphi_G := \varphi_{G,G_n}$. 
Then the map 
$\varphi_X \circ \alpha : G \times X \to X/G_n$
is invariant under the natural action of $G_n \times G_n$, since
we have for any scheme $S$ and any $u,v \in G_n(S)$,
$g \in G(S)$, $x \in X(S)$ that $(ug)(vx) = u (g v g^{-1}) g x$
and $g v g^{-1} \in G_n(S)$. Also, the map
\[ \varphi_G \times \varphi_X : G \times X 
\longrightarrow G/G_n \times X/G_n \]
is the categorical quotient by $G_n \times G_n$ in view
of Lemma \ref{lem:prod}. Thus, there exists a unique morphism
$\beta : G/G_n \times X/G_n \to X/G_n$
such that the following diagram commutes:
\[ 
\xymatrix{G \times X \ar[r]^{\alpha} \ar[d]_{\varphi_G \times \varphi_X} 
& X \ar[d]^{\varphi_X} \\
G/G_n \times X/G_n \ar[r]^-{\beta} & X/G_n. \\
}
\]
We have in particular 
$\beta(e_{G/G_n},\varphi_X(x)) = \varphi_X(x)$
for any schematic point $x$ of $X$. As $\varphi_X$
is an epimorphism (Remark \ref{rem:qi}), it follows that
$\beta(e_{G/G_n}, z) = z$ for any schematic point $z$ of
$X/G_n$. Likewise, we obtain 
$\beta(x,\beta(y,z)) = \beta(xy,z)$ for any schematic points
$x,y$ of $G/G_n$ and $z$ of $X/G_n$, by using the fact that 
\[ \varphi_G \times \varphi_G \times \varphi_X : 
G \times G \times X \longrightarrow 
G/G_n \times G/G_n \times X/G_n \]
is an epimorphism (as follows from Lemma \ref{lem:prod}
and Remark \ref{rem:qi} again). Thus, $\beta$ is the desired 
action.
\end{proof}

\begin{lemma}\label{lem:lin}
Let $G$ be a connected affine algebraic group, $X$ a normal 
$G$-variety, and $L$ a line bundle on $X$. Then 
$L^{\otimes m }$ is $G$-linearizable for some positive
integer $m$ depending only on $G$.
\end{lemma}

\begin{proof}
If $G$ is smooth, then the assertion is that of 
\cite[Thm.~2.14]{Brion15}. For an arbitrary $G$, we may choose 
a positive integer $n$ such that $G/G_n$ is smooth. 
In view of Lemmas \ref{lem:quot} and \ref{lem:equiv}, 
the categorical quotient $X/G_n$ is a normal $G/G_n$-variety equipped 
with a $G$-equivariant morphism $\varphi : X \to X/G_n$ and with a line
bundle $M$ such that $\varphi^*(M) \cong L^{\otimes p^n}$. The line
bundle $M^{\otimes m}$ is $G/G_n$-linearizable for some positive
integer $m$, and hence $L^{\otimes p^n m}$ is $G$-linearizable.
\end{proof}

\subsection{$G$-quasi-projectivity}
\label{subsec:Gqp}

We say that a $G$-scheme $X$ is \emph{$G$-quasi-projective} 
if it admits an ample $G$-linearized line bundle; equivalently,
$X$ admits an equivariant immersion in the projectivization
of a finite-dimensional $G$-module. If in addition the $G$-action
on $X$ is almost faithful, then $G$ must be affine, since it acts
almost faithfully on a projective space.

By the next lemma, being $G$-quasi-projective is invariant 
under field extensions (this fact should be well-known, but 
we could not locate any reference):

\begin{lemma}\label{lem:field}
Let $G$ be an algebraic $k$-group, $X$ a $G$-scheme over $k$, 
and $k'/k$ a field extension. Then $X$ is $G$-quasi-projective
if and only if $X_{k'}$ is $G_{k'}$-quasi-projective.
\end{lemma}

\begin{proof}
Assume that $X$ has an ample $G$-linearized line bundle $L$.
Then $L_{k'}$ is an ample line bundle on $X_{k'}$
(see \cite[II.4.6.13]{EGA}), and is $G_{k'}$-linearized 
by Example \ref{ex:fp} (i).

For the converse, we adapt a classical specialization argument 
(see \cite[IV.9.1]{EGA}). Assume that $X_{k'}$ has an ample
$G_{k'}$-linearized line bundle $L'$. Then there exists
a finitely generated subextension $k''/k$ of $k'/k$ and a line bundle 
$L''$ on $X_{k''}$ such that $L' \cong L'' \otimes_{k''} k'$; moreover,
$L''$ is ample in view of \cite[VIII.5.8]{SGA1}. 
We may further assume (possibly by enlarging $k''$) 
that $L''$ is $G_{k''}$-linearized. Next, there exists a finitely
generated $k$-algebra $R \subset k''$ and an ample line bundle 
$M$ on $X_R$ such that $L'' \cong M \otimes_R k''$ and 
$M$ is $G_R$-linearized. Choose a maximal ideal 
$\fm \subset R$, with quotient field $K := R/\fm$. Then $K$ is a finite
extension of $k$; moreover, $X_K$ is equipped with an ample
$G_K$-linearized line bundle $M_K := M \otimes_R K$. Consider 
the norm $L := \N(M_K)$; then $L$ is an ample line bundle on $X$ 
in view of \cite[II.6.6.2]{EGA}. Also, $L$ is equipped with a 
$G$-linearization by Lemma \ref{lem:nwr}.
\end{proof}

Also, $G$-quasi-projectivity is invariant under Frobenius twists
(Lemma \ref{lem:ifm}) and quotients by infinitesimal group
schemes (Lemmas \ref{lem:quot} and \ref{lem:equiv}). We will
obtain a further invariance property of quasi-projectivity
(Proposition \ref{prop:GH}). For this, we need some preliminary 
notions and results.

Let $G$ be an algebraic group, $H \subset G$ a subgroup scheme,
and $Y$ an $H$-scheme. The \emph{associated fiber bundle} 
is a $G$-scheme $X$ equipped with a $G \times H$-equivariant 
morphism $\varphi : G \times Y \to X$ such that the square  
\[
\xymatrix{
G \times Y \ar[r]^-{\pr_G} \ar[d]_{\varphi} & G \ar[d]_f \\
X \ar[r]^-{\psi} & G/H, \\
}
\]
is cartesian, where $f$ denotes the quotient morphism, 
$\pr_G$ the projection, and $G \times H$ acts on $G \times Y$ via 
$(g,h) \cdot (g',y) = (gg' h^{-1}, h \cdot y)$ for any
scheme $S$ and any $g,g' \in G(S)$, $h \in H(S)$, $y \in Y(S)$. 
Then $\varphi$ is an $H$-torsor, since so is $f$. Thus, the triple 
$(X,\varphi,\psi)$ is uniquely determined; we will denote $X$ by
$G \times^H Y$. Also, note that $\psi$ is faithfully flat and 
$G$-equivariant; its fiber at the base point $f(e_G) \in (G/H)(k)$ 
is isomorphic to $Y$ as an $H$-scheme. 

Conversely, if $X$ is a $G$-scheme equipped with an equivariant
morphism $\psi: X \to G/H$, then $X = G \times^H Y$, where
$Y$ denotes the fiber of $\psi$ at the base point of $G/H$. 
Indeed, form the cartesian square
\[
\xymatrix{
X' \ar[r]^-{\eta} \ar[d]_{\varphi} & G \ar[d]_f \\
X \ar[r]^-{\psi} & G/H. \\
}
\]
Then $X'$ is a $G$-scheme, and $\eta$ an equivariant morphism
for the $G$-action on itself by left multiplication.
Moreover, we may identify $Y$ with the fiber of $\eta$ 
at $e_G$. Then $X'$ is equivariantly isomorphic to $G \times Y$
via the maps 
$G \times Y \to X'$, $(g,y) \mapsto g \cdot y$ and 
$X' \to G \times Y$, 
$z \mapsto (\psi'(z), \psi'(z)^{-1} \cdot z)$, and this
identifies $\eta$ with $\pr_G : G \times Y \to G$.

The associated fiber bundle need not exist in general, as follows 
from Hironaka's example mentioned in the introduction (see 
\cite[p.~367]{BB} for details). But it does exist when the $H$-action 
on $Y$ extends to a $G$-action $\alpha : G \times Y \to Y$: 
just take $X = G/H \times Y$ equipped with the diagonal action 
of $G$ and with the maps 
\[ f \times \alpha : G \times Y \longrightarrow G/H \times Y, 
\quad \pr_{G/H} : G/H \times Y \longrightarrow G/H. \]
A further instance in which the associated fiber bundle exists
is given by the following result, which follows from 
\cite[Prop.~7.1]{MFK}:

\begin{lemma}\label{lem:ass}
Let $G$ be an algebraic group, $H \subset G$ a subgroup scheme,
and $Y$ an $H$-scheme equipped with an ample $H$-linearized
line bundle $M$. Then the associated fiber bundles 
$X := G \times^H Y$ and $L := G \times^H M$ exist. Moreover,
$L$ is a $G$-linearized line bundle on $X$, and is ample relative
to $\psi$. In particular, $X$ is quasi-projective.
\end{lemma}

In particular, the associated fiber bundle $G \times^H V$ exists
for any finite-dimensional $H$-module $V$, viewed as an
affine space. Then $G\times^H V$ is a $G$-linearized vector
bundle on $G/H$, called the \emph{homogeneous vector bundle}
associated with the $H$-module $V$. 

We now come to a key technical result:

\begin{proposition}\label{prop:GH}
Let $G$ be an algebraic group, $H \subset G$ a subgroup scheme 
such that $G/H$ is finite, and $X$ a $G$-scheme. If $X$ is 
$H$-quasi-projective, then it is $G$-quasi-projective as well.
\end{proposition}

\begin{proof}
We first reduce to the case where \emph{$G$ is smooth}. 
For this, we may assume that $\car(k) = p > 0$. 
Choose a positive integer $n$ such that $G/G_n$ is smooth;
then we may identify $H/H_n$ with a subgroup scheme of
$G/G_n$, and the quotient $(G/G_n)/(H/H_n)$ is finite. 
By Lemma \ref{lem:ifm}, $X^{(p^n)}$ is a $G/G_n$-scheme 
of finite type and admits an ample $H/H_n$-linearized line bundle. 
If $X^{(p^n)}$ admits an ample $G/G_n$-linearized line bundle $M$, 
then $(F^n_{X/k})^*(M)$ is an ample $G$-linearized line bundle on $X$,
in view of Lemma \ref{lem:ifm} again. This yields the desired reduction.

Next, let $M$ be an ample $H$-linearized line bundle on $X$. 
By Lemma \ref{lem:ass}, the associated fiber bundle 
$G \times^H X = G/H \times X$ is equipped with the $G$-linearized 
line bundle $L := G \times^H M$.  The projection 
$\pr_X : G/H \times X \to X$ is finite, \'etale of degree 
$n := [G:H]$, and $G$-equivariant. As a consequence, 
$E := (\pr_X)_*(L)$ is a $G$-linearized vector bundle of degree 
$n$ on $X$; thus, $\det(E)$ is $G$-linearized as well. To complete 
the proof, it suffices to show that $\det(E)$ is ample.

For this, we may assume that $k$ is algebraically closed
by using \cite[VIII.5.8]{SGA1} again. Then there exist lifts 
$e = g_1, \ldots, g_n \in G(k)$ of the distinct $k$-points of $G/H$. 
This identifies $G/H \times X$ with the disjoint union of 
$n$ copies of $X$; the pull-back of $\pr_X$ to the $i$th copy is 
the identity of $X$, and the pull-back of $L$ is $g_i^*(M)$. Thus, 
$E \cong \oplus_{i = 1}^n g_i^*(M)$, and hence
$\det(E) \cong \otimes_{i=1}^n g_i^*(M)$ is ample indeed.
\end{proof}

\begin{remark}\label{rem:GH}
Given $G$, $H$, $X$ as in Proposition \ref{prop:GH} and an
$H$-linearized ample line bundle $M$ on $X$, it may well happen
that no non-zero tensor power of $M$ is $G$-linearizable. This holds 
for example when $G$ is the constant group of order $2$ acting on 
$X = \bP^1 \times \bP^1$ by exchanging both factors, $H$ is trivial, 
and $M$ has bi-degree $(m_1,m_2)$ with $m_1 > m_2 \ge 1$. 
\end{remark}

\begin{corollary}\label{cor:GH}
Let $G$ be an affine algebraic group, and $X$ a normal 
quasi-projective $G$-variety. Then $X$ is $G$-quasi-projective.
\end{corollary}

\begin{proof}
Choose an ample line bundle $L$ on $X$. By Lemma \ref{lem:lin}, 
some positive power of $L$ admits a $G^0$-linearization. 
This yields the assertion in view of Proposition \ref{prop:GH}. 
\end{proof}

\section{Proofs of the main results}
\label{subsec:pr}

\subsection{The theorem of the square}
\label{subsec:ts}

Let $G$ be a group scheme with multiplication map 
$\mu$, and $X$ a $G$-scheme with action map $\alpha$.
For any line bundle $L$ on $X$, denote by $L_G$ the 
line bundle on $G \times X$ defined by
\[ L_G := \alpha^*(L) \otimes \pr_X^*(L)^{-1},\]
where $\pr_X : G \times X \to X$ stands for the projection.
Next, denote by $L_{G \times G}$ the line bundle on 
$G \times G \times X$ defined by
\[ L_{G \times G} := (\mu \times \id_X)^*(L_G)
\otimes (\pr_1 \times \id_X)^*(L_G)^{-1} 
\otimes (\pr_2 \times \id_X)^*(L_G)^{-1}, \]
where $\pr_1, \pr_2 : G \times G \to G$ denote the two projections. 
Then $L$ is said to \emph{satisfy the theorem of the square} 
if there exists a line bundle $M$ on $G \times G$ such that 
\[ L_{G \times G} \cong \pr_{G \times G}^*(M), \]
where $\pr_{G \times G} : G \times G \times X \to G \times G$
denotes the projection.

By \cite[p.~159]{BLR}, $L$ satisfies the theorem of the square 
if and only if the polarization morphism
\[ G \longrightarrow \Pic_X, \quad 
g \longmapsto g^*(L) \otimes L^{-1} \]
is a homomorphism of group functors, where $\Pic_X$
denotes the Picard functor that assigns with any scheme 
$S$, the commutative group $\Pic(X \times S)/\pr_S^*\Pic(S)$. 
In particular, the line bundle
$(gh)^*(L) \otimes g^*(L)^{-1} \otimes h^*(L)^{-1} \otimes L$
is trivial for any $g,h \in G(k)$; this is the original formulation 
of the theorem of the square. 

\begin{proposition}\label{prop:ts}
Let $G$ be a connected algebraic group, $X$ a normal, geometrically
irreducible $G$-variety, and $L$ a line bundle on $X$. 
Then $L^{\otimes m}$ satisfies the theorem of the square
for some positive integer $m$ depending only on $G$.
\end{proposition} 

\begin{proof}
By a generalization of Chevalley's structure theorem due 
to Raynaud (see \cite[IX.2.7]{Raynaud} and also 
\cite[9.2 Thm.~1]{BLR}), 
there exists an exact sequence of algebraic groups
\[ 1 \longrightarrow H \longrightarrow G 
\stackrel{f}{\longrightarrow} A \longrightarrow 1, \]
where $H$ is affine and connected, and $A$ is an abelian variety. 
(If $G$ is smooth, then there exists a smallest such subgroup scheme 
$H = H(G)$; if in addition $k$ is perfect, then $H(G)$ is smooth as well). 
We choose such a subgroup scheme $H \triangleleft G$.

In view of Lemma \ref{lem:lin}, there exists a positive integer $m$
such that $L^{\otimes m}$ is $H$-linearizable. Replacing $L$ with
$L^{\otimes m}$, we may thus assume that $L$ is equipped with an
$H$-linearization. Then $L_G$ is also $H$-linearized for the
action of $H$ on $G \times X$ by left multiplication on $G$,
since $\alpha$ is $G$-equivariant for that action, and 
$\pr_X$ is $G$-invariant. As the map
$f \times \id_X : G \times X \to A \times X$ is an $H$-torsor
relative to the above action, there exists a line bundle
$L_A$ on $A \times X$, unique up to isomorphism, such that
\[ L_G = (f \times \id_X)^*(L_A) \]
(see \cite[p.~32]{MFK}). The diagram
\[ 
\xymatrixcolsep{4pc}\xymatrix{
G \times G \times X \ar[r]^-{\mu_G \times \id_X} \ar[d]_{f \times f \times \id_X} 
& G \times X \ar[d]^{f \times \id_X} \\ 
A \times A \times X \ar[r]^-{\mu_A\times \id_X} & A \times X  \\}
\]  
commutes, since $f$ is a morphism of algebraic groups; thus,
\[ (\mu_G \times \id_X)^*(L_G) \cong 
(f \times f \times \id_X)^* (\mu_A \times \id_X)^*(L_A). \]
Also, for $i = 1,2$, the diagrams
\[ 
\xymatrixcolsep{4pc}\xymatrix{
G \times G \times X \ar[r]^-{\pr_i \times \id_X} \ar[d]_{f \times f \times \id_X} 
& G \times X \ar[d]^{f \times \id_X} \\ 
A \times A \times X \ar[r]^-{\pr_i \times \id_X} & A \times X  \\}
\]  
commute as well, and hence 
\[ (\pr_i \times \id_X)^*(L_G) \cong 
(f \times f \times \id_X)^* (\pr_i \times \id_X)^*(L_A). \]
This yields an isomorphism
\[ L_{G \times G} \cong (f \times f \times \id_X)^*(L_{A \times A}), \]
where we set
\[ L_{A \times A} := (\mu_A \times \id_X)^*(L_A)
\otimes (\pr_1 \times \id_X)^*(L_A)^{-1} 
\otimes (\pr_2 \times \id_X)^*(L_A)^{-1}. \]

Note that the line bundle $L_A$ on $A \times X$ is equipped
with a rigidification along $e_A \times X$, i.e., with an isomorphism
\[ \cO_X \stackrel{\cong}{\longrightarrow} 
(e_A \times \id_X)^*(L_A). \]
Indeed, recall that 
$L_A = \alpha^*(L) \otimes \pr_X^*(L)^{-1}$ and 
$\alpha \circ (e_A \times \id_X) = \pr_X \circ (e_A \times \id_X)$.
As 
$(\mu_A \times \id_X) \circ (e_A \times \id_{A \times X}) 
= \pr_2 \circ (e_A \times \id_{A \times X})$ and 
$(\pr_1 \times \id_X) \circ (e_A \times \id_{A \times X}) 
= e_A \times  \id_X$,
it follows that $L_{A \times A}$ is equipped with a rigidification
along $e_A \times A \times X$. Likewise, $L_{A \times A}$
is equipped with a rigidification along $A \times e_A \times X$. 
The assertion now follows from the lemma below, a version of the 
classical theorem of the cube (see \cite[III.10]{Mumford}).
\end{proof}

\begin{lemma}\label{lem:cube}
Let $X$, $Y$ be proper varieties equipped with $k$-rational points 
$x$, $y$. Let $Z$ be a geometrically connected scheme of finite type, 
and $L$ a line bundle on $X \times Y \times Z$. Assume that the 
pull-backs of $L$ to $x \times Y \times Z$ and $X \times y \times Z$ 
are trivial. Then $L \cong \pr_{X \times Y}^*(M)$ for some line bundle 
$M$ on $X \times Y$.
\end{lemma}

\begin{proof}
By our assumptions on $X$ and $Y$, we have $\cO(X) = k = \cO(Y)$.
Choose rigidifications 
\[ \cO_{Y \times Z} \stackrel{\cong}{\longrightarrow} 
(x \times \id_Y \times \id_Z)^*(L), \quad 
\cO_{X \times Z} \stackrel{\cong}{\longrightarrow} 
(\id_X \times y \times \id_Z)^*(L). \]
We may assume that these rigidifications induce the same isomorphism 
\[ \cO_Z \stackrel{\cong}{\longrightarrow} 
(x \times y \times \id_Z)^*(L),
\] 
since their pull-backs to $Z$ differ by a unit in 
$\cO(Z) = \cO(Y \times Z) = \cO(X \times Z)$.

By \cite[II.15]{Murre} together with \cite[Thm.~2.5]{Kleiman}, 
the Picard functor $\Pic_X$ is represented by a commutative group 
scheme, locally of finite type, and likewise for 
$\Pic_Y$, $\Pic_{X \times Y}$. Also, we may view
$\Pic_{X \times Y}(Z)$ as the group of isomorphism classes of 
line bundles on $X \times Y \times Z$, rigidified along 
$x \times y \times Z$, and likewise for $\Pic_X(Z)$, $\Pic_Y(Z)$
(see e.g. \cite[Lem.~2.9]{Kleiman}).  
Thus, $L$ defines a morphism of schemes
\[ \varphi : Z \longrightarrow \Pic_{X \times Y}, \quad
z \longmapsto (\id_{X \times Y} \times z)^*(L). \]

Denote by $N$ the kernel of the morphism of group schemes
\[ \pr_X^* \times \pr_Y^* : \Pic_{X \times Y} 
\longrightarrow \Pic_X \times \Pic_Y. \]
Then $\varphi$ factors through $N$ in view of the rigidifications
of $L$. We now claim that $N$ is \'etale. To check this,
it suffices to show that the differential of 
$\pr_X^* \times \pr_Y^*$ at the origin is an isomorphism. 
But we have
\[ \Lie(\Pic_{X \times Y}) \cong H^1(X \times Y, \cO_{X \times Y})
\cong (H^1(X, \cO_X) \otimes \cO(Y)) \oplus 
(\cO(X) \otimes H^1(Y,\cO_Y)). \]
where the first isomorphism follows from \cite[Thm.~5.11]{Kleiman}, 
and the second one from the K\"unneth formula. Thus,
\[ \Lie(\Pic_{X \times Y}) \cong H^1(X, \cO_X) \oplus H^1(Y,\cO_Y) 
\cong \Lie(\Pic_X) \oplus \Lie(\Pic_Y). \]
Moreover, these isomorphisms identify the differential of 
$\pr_X^* \times \pr_Y^*$ at the origin with the identity. 
This implies the claim.

Since $Z$ is geometrically connected, it follows from the claim 
that $\varphi$ factors through a $k$-rational point of $N$. 
By the definition of the Picard functor, this means that 
\[ L \cong \pr_{X \times Y}^*(M) \otimes \pr_Z^*(M') \]
for some line bundles $M$ on $X \times Y$ and $M'$ on $Z$.
Using again the rigidifications of $L$, we see that
$M'$ is trivial. 
\end{proof}

\subsection{Proof of Theorem \ref{thm:cover}}
\label{subsec:cover}

Let $X$ be a normal $G$-variety, where $G$ is a connected
algebraic group. We first reduce to the case where $G$ is
\emph{smooth}; for this, we may assume that $\car(k) > 0$.
By Lemmas \ref{lem:quot} and \ref{lem:equiv},
there is a finite $G$-equivariant morphism
$\varphi : X \to X/G_n$ for all $n \geq 1$, where $X/G_n$
is a normal $G/G_n$-variety. Since $G/G_n$ is smooth
for $n \gg 0$, this yields the desired reduction.


Consider an open affine subvariety $U$ of $X$. Then the 
image $G \cdot U = \alpha(G \times U)$ is open in $X$ (since
$\alpha$ is flat), and $G$-stable. Clearly, $X$ is covered by 
opens of the form $G \cdot U$ for $U$ as above; thus, it suffices 
to show that $G \cdot U$ is quasi-projective. This follows from 
the next proposition, a variant of a result of Raynaud on
the quasi-projectivity of torsors (see \cite[V.3.10]{Raynaud} 
and also \cite[6.4 Prop.~2]{BLR}).

\begin{proposition}\label{prop:ray}
Let $G$ be a smooth connected algebraic group, $X$ a normal $G$-variety,
and $U \subset X$ an open affine subvariety. Assume that $X = G \cdot U$ 
and let $D$ be an effective Weil divisor on $X$ with support
$X \setminus U$. Then $D$ is an ample Cartier divisor.
\end{proposition}
   
\begin{proof}
By our assumptions on $G$, the action map 
$\alpha : G \times X \to X$ is smooth and its fibers are geometrically
irreducible; in particular, $G \times X$ is normal. Also, 
the Weil divisor $G \times D$ on $G \times X$ contains no fiber 
of $\alpha$ in its support, since $X = G \cdot U$. In view of the
Ramanujam-Samuel theorem (see \cite[IV.21.14.1]{EGA}), it follows 
that $G \times D$ is a Cartier divisor. As $D$ is the pull-back 
of $G \times D$ under $e_G \times \id_X$, we see that $D$ is Cartier.

To show that $D$ is ample, we may replace $k$ with any separable 
field extension, since normality is preserved under such extensions.
Thus, we may assume that $k$ is \emph{separably closed}.
By Proposition \ref{prop:ts}, there exists a positive integer $m$ 
such that the line bundle on $X$ associated with $mD$ satisfies 
the theorem of the square. Replacing $D$ with $mD$, we see that 
the divisor $gh \cdot D  - g \cdot D - h \cdot D + D$ is principal 
for all $g, h \in G(k)$. In particular, we have isomorphisms
\[ \cO_X(2D) \cong \cO_X(g \cdot D + g^{-1} \cdot D) \]
for all $g \in G(k)$.

We now adapt an argument from \cite[p.~154]{BLR}. In view of the
above isomorphism, we have global sections $s_g \in H^0(X,\cO_X(2D))$ 
($g \in G(k))$ such that $X_{s_g} = g \cdot U \cap g^{-1} \cdot U$ 
is affine. Thus, it suffices to show that $X$ is covered by the 
$g \cdot U \cap g^{-1} \cdot U$, where $g \in G(k)$. In turn, it suffices 
to check that every closed point $x \in X$ lies in 
$g \cdot U \cap g^{-1} \cdot U$ for some $g \in G(k)$.

Denote by $k'$ the residue field of $x$; this is a finite
extension of $k$. Consider the orbit map
\[ \alpha_x : G_{k'} \longrightarrow X_{k'}, \quad
g \longmapsto g \cdot x. \]
Then $V := \alpha_x^{-1}(U_{k'})$ is open in $G_{k'}$, and
non-empty as $X = G \cdot U$. Since $G$ is geometrically irreducible,
$V \cap V^{-1}$ is open and dense in $G_{k'}$. As 
$\pr_G : G_{k'} \to G$ is finite and surjective, there exists 
a dense open subvariety $W$ of $G$ such that 
$W_{k'} \subset V \cap V^{-1}$. Also, since $G$ is smooth, $G(k)$
is dense in $G$, and hence $W(k)$ is non-empty. Moreover,
$x \in g \cdot U \cap g^{-1} \cdot U$ for any $g \in G(k)$.
\end{proof}

\subsection{Proof of Theorem \ref{thm:model}}
\label{subsec:model}

It suffices to show that $X$ is $G$-equivariantly isomorphic to 
$G \times^H Y$ for some subgroup scheme $H\subset G$ such
that $G/H$ is an abelian variety, and some $H$-quasi-projective
closed subscheme $Y \subseteq X$. Indeed, we may then view 
$Y$ as a $H$-stable subscheme of the projectivization $ \bP(V)$
of some finite-dimensional $H$-module $V$. Hence $X$ is 
a $G$-stable subscheme of the projectivization $\bP(E)$,
where $E$ denotes the homogeneous vector bundle 
$G \times^H V \to G/H$.

Next, we reduce to the case where $G$ is \emph{smooth},
as in the proof of Theorem \ref{thm:cover}.
We may of course assume that $\car(k) > 0$.
Choose a positive integer $n$ such that $G/G_n$ is smooth
and recall from Lemmas \ref{lem:quot} and \ref{lem:equiv} that 
$X/G_n =: X'$  is a normal quasi-projective variety equipped 
with an action of $G/G_n =: G'$ such that the quotient morphism 
$\varphi : X \to X'$ is equivariant. Assume that there exists an
equivariant isomorphism $X' \cong G' \times^{H'} Y'$ satisfying
the above conditions. Let $H \subset G$ (resp. $Y \subset X$) 
be the subgroup scheme (resp. the closed subscheme)
obtained by pulling back $H' \subset G'$ (resp. $Y' \subset X'$). 
Then $G/H \cong G'/H'$, and hence the composition
$X \to X' \to G'/H'$ is a $G$-equivariant morphism with fiber
$Y$ at the base point. This yields a $G$-equivariant isomorphism
$X \cong  G \times^H Y$, where $G/H$ is an abelian variety.
Moreover, $Y$ is $H$-quasi-projective, since it is equipped with
a finite $H$-equivariant morphism to $Y'$, and the latter is 
$H$-quasi-projective. This yields the desired reduction.

Replacing $G$ with its quotient by the kernel of the action, we may
further assume that \emph{$G$ acts faithfully on $X$}. 
We now use the notation of the proof of Proposition \ref{prop:ts}; 
in particular, we choose a normal connected affine subgroup scheme
$H \triangleleft G$ such that $G/H$ is an abelian variety, and an ample 
$H$-linearized line bundle $L$ on $X$. Recall that the line bundle 
$L_G = \alpha^*(L) \otimes \pr_X^*(L^{-1})$
satisfies $L_G = (f \times \id_X)^*(L_A)$ for a line bundle $L_A$ on 
$A \times X$, rigidified along $e_A \times X$. Since the Picard 
functor $\Pic_A$ is representable, this yields a morphism of schemes
\[ \varphi : X \longrightarrow \Pic_A, \quad
x \longmapsto (\id_A \times x)^*(L_A). \]

We first show that \emph{$\varphi$ is $G$-equivariant} relative
to the given $G$-action on $X$, and the $G$-action on $\Pic_A$ via
the morphism $f: G \to A$ and the $A$-action on $\Pic_A$ by
translation. Since $G \times X$ is reduced (as $G$ is smooth and 
$X$ is reduced), it suffices to check the equivariance on points
with values in fields. So let $k'/k$ be a field extension, and
$g \in G(k')$, $x \in X(k')$. Then 
$\varphi(x) \in \Pic_A(k') = \Pic(A_{k'})$. 
Moreover, by \cite[p.~32]{MFK}, the pull-back map $f_{k'}^*$ 
identifies $\Pic(A_{k'})$ with 
the group $\Pic^{H_{k'}}(G_{k'})$ of $H_{k'}$-linearized
line bundles on $G_{k'}$; also,
$f_{k'}^* \varphi_{k'}(x) = (\id_{G_{k'}} \times x)^*(L_{G_{k'}})$
in $\Pic^{H_{k'}}(G_{k'})$. Thus, 
\[ f_{k'}^* \varphi_{k'}(x) = \alpha_x^*(L_{k'}), \] 
where $\alpha_x : G_{k'} \to X_{k'}$ denotes the orbit map. 
We have $\alpha_{g \cdot x} = \alpha_x \circ \rho(g)$,
where $\rho(g)$ denotes the right multiplication by $g$ in 
$G_{k'}$. Hence
\[ f_{k'}^* \varphi_{k'}(g \cdot x) = \rho(g)^* f_{k'}^* \varphi_{k'}(x). \]
Also, since $f$ is $G$-equivariant, we have 
$f_{k'} \circ \rho(g) = \tau(f_{k'}(g)) \circ f_{k'}$,
where $\tau(a)$ denotes the translation by $a \in A(k')$ in the abelian
variety $A_{k'}$. This yields the equality 
\[ f_{k'}^* \varphi_{k'}(g \cdot x) = 
f_{k'}^* \tau(f_{k'}(g))^* \varphi_{k'}(x) \]
in $\Pic^{H_{k'}}(G_{k'})$, and hence the desired equality
\[ \varphi_{k'}(g \cdot x) = \tau(f_{k'}(g))^* \varphi_{k'}(x) \]
in $\Pic(A_{k'})$. 

Next, we show that \emph{$\varphi(x)$ is ample for any $x \in X$}.
By \cite[XI.1.11.1]{Raynaud}, it suffices to show that the line
bundle $f^*\varphi(x)$ on $G_{k'}$ is ample, where $k'$ is as above;
equivalently, $\alpha_x^*(L)$ is ample. The orbit map $\alpha_x$
(viewed as a morphism from $G$ to the orbit of $x$) 
may be identified with the quotient map by the isotropy subgroup
scheme $G_{k',x} \subset G_{k'}$. This subgroup scheme is affine (see 
e.g. \cite[Prop.~3.1.6]{Brion17}) and hence so is the morphism 
$\alpha_x$. As $L$ is ample, this yields the assertion.

Now recall the exact sequence of group schemes
\[ 0 \longrightarrow \hat{A} \longrightarrow \Pic_A
\longrightarrow \NS_A \longrightarrow 0, \]
where $\hat{A} = \Pic^0_A$ denotes the dual abelian variety,
and $\NS_A = \pi_0(\Pic_A)$ the N\'eron-Severi group scheme; 
moreover, $\NS_A$ is \'etale. 

\emph{If $X$ is geometrically irreducible}, it follows that the base 
change $\varphi_{k_s} : X_{k_s} \to \Pic_{A_{k_s}}$ factors 
through a unique coset $Y = \hat{A}_{k_s} \cdot M$, where 
$M$ is an ample line bundle on $A_{k_s}$. We then have an 
$A_{k_s}$-equivariant isomorphism $Y \cong A_{k_s}/K(M)$, where
$K(M)$ is a finite subgroup scheme of $A_{k_s}$. So there exists
a finite Galois extension $k'/k$ and a $G_{k'}$-equivariant
morphism of $k'$-schemes $\varphi': X_{k'} \to A_{k'}/F$,
where $F$ is a finite subgroup scheme of $A_{k'}$. As $F$ is
contained in the $n$-torsion subgroup scheme $A_{k'}[n]$
for some positive integer $n$, and 
$A_{k'}/A_{k'}[n] \stackrel{\cong}{\longrightarrow} A_{k'}$
via the multiplication by $n$ in $A_{k'}$, we obtain a
morphism of $k'$-schemes $\varphi'': X_{k'} \to A_{k'}$
which satisfies the equivariance property 
\[ \varphi''(g \cdot x) = \tau(n f(g)) \cdot \varphi''(x) \]
for all schematic points $g \in G_{k'}$, $x \in X_{k'}$.

The Galois group $\Gamma_{k'} := \Gal(k'/k)$ acts on 
$G_{k'}$ and $A_{k'}$; replacing $\varphi''$ with the sum
of its $\Gamma_{k'}$-conjugates (and $n$ with $n [k':k]$),
we may assume that $\varphi''$ is $\Gamma_{k'}$-equivariant.
Thus, $\varphi''$ descends to a morphism 
$\psi :X \to A$ such that 
$\psi(g \cdot x) = \tau(n f(g)) \cdot \psi(x)$
for all schematic points $g \in G$, $x \in X$. We may view 
$\psi$ as a $G$-equivariant morphism to $A/A[n]$, or equivalently, 
to $G/H'$, where $H' \subset G$ denotes the pull-back
of $A[n] \subset A$ under $f$. Since $H'/H$ is finite, we see that
$G/H'$ is an abelian variety and $H'$ is affine. Moreover, $\psi$
yields a $G$-equivariant isomorphism $X \cong G\times^{H'} Y$
for some closed $H'$-stable subscheme $Y \subset X$.
By Corollary \ref{cor:GH}, $X$ is $H'$-quasi-projective; hence so
is $Y$. This completes the proof in this case.

Finally, we consider the \emph{general case}, where $X$ is not 
necessarily geometrically irreducible. By Example \ref{ex:fp} (iv), 
we may view $X$ as a geometrically irreducible $K$-variety,
where $K$ denotes the separable algebraic closure of $k$ in
$k(X)$. Moreover, $G_K$ acts faithfully on $X$ via $\pr_G : G_K \to G$. 
Also, $X$ is quasi-projective over $K$ in view of 
\cite[II.6.6.5]{EGA}. So the preceding step yields a $G_K$-equivariant
morphism $X \to G_K/H'$ for some normal affine $K$-subgroup scheme
$H' \triangleleft G_K$ such that $A' = G_K/H'$ is an abelian 
variety. On the other hand, we have an exact sequence of $K$-group
schemes
\[ 1 \longrightarrow H_K \longrightarrow G_K 
\stackrel{f_K}{\longrightarrow} A_K  \longrightarrow 1, \]
where $H_K$ is affine and $A_K$ is an abelian variety. Consider
the subgroup scheme $H_K \cdot H' \subset G_K$ generated by $H_K$
and $H'$. Then $H_K \cdot H'/H' \cong H_K/H_K \cap H'$ is
affine (as a quotient group of $H_K$) and proper (as a subgroup
scheme of $G_K/H_K = A_K$), hence finite. Thus, $H_K \cdot H'$
is affine, and the natural map $G_K/H' \to G_K/H_K \cdot H'$ is 
an isogeny of abelian varieties. Replacing $H'$ with $H_K \cdot H'$, 
we may therefore assume that $H_K \subset H'$. Then the finite subgroup
scheme $H'/H_K \subset A_K$ is contained in $A_K[n]$ for some
positive integer $n$. This yields a $G_K$-equivariant morphism
$X \to A_K/A_K[n]$, and hence a $G$-equivariant morphism
$X \to A/A[n]$ by composing with 
$\pr_{A/A[n]}: A_K/A_K[n] \to A/A[n]$. Arguing as at the end of the
preceding step completes the proof. 

\begin{remarks}\label{rem:fin}
(i) Consider a smooth connected algebraic group $G$
and an affine subgroup scheme $H$ such that $G/H$ is an
abelian variety. Then the quotient map $G \to G/H$ is a
morphism of algebraic groups (see e.g. \cite[Prop.~4.1.4]{Brion17}).
In particular, $H$ is normalized by $G$. But in positive characteristics,
this does not extend to an arbitrary connected algebraic group $G$. 
Consider indeed a non-trivial abelian variety $A$; then we
may choose a non-trivial infinitesimal subgroup $H \subset A$, and 
form the semi-direct product $G := H \ltimes A$, where $H$ acts on 
$A$ by translation. So $H$ is identified with a non-normal subgroup
of $G$ such that the quotient $G/H = A$ is an abelian variety.

Also, recall that a smooth connected algebraic group $G$
admits a smallest (normal) subgroup scheme with quotient
an abelian variety. This also fails for non-smooth 
algebraic groups, in view of \cite[Ex.~4.3.8]{Brion17}.

\medskip

\noindent
(ii) With the notation and assumptions of Theorem \ref{thm:model},
we have seen that $X$ is an associated fiber bundle $G \times^H Y$ 
for some subgroup scheme $H \subset G$ such that $G/H$ is
an abelian variety, and some closed $H$-quasi-projective subscheme
$Y \subset X$. If $G$ acts almost faithfully on $X$, then the
$H$-action on $Y$ is almost faithful as well; thus, $H$ is affine.

Note that the pair $(H,Y)$ is not uniquely determined by $(G,X)$,
since $H$ may be replaced with any subgroup scheme $H' \subset G$
such that $H' \supset H$ and $H'/H$ is finite. So one may rather
ask whether there exists such a pair $(H,Y)$ with a smallest
subgroup scheme $H$, i.e., the corresponding morphism
$\psi : X \to G/H$ is universal among all such morphisms. 
The answer to this question is generally negative (see 
\cite[Ex.~5.1]{Brion10}); yet one can show that it is positive
in the case where $G$ is smooth and $X$ is almost homogeneous 
under $G$. 

Even under these additional assumptions, there may exist no pair 
$(H,Y)$ with $H$ smooth or $Y$ geometrically reduced. Indeed, assume 
that $k$ is imperfect; then as shown by Totaro (see \cite{Totaro}), 
there exist non-trivial \emph{pseudo-abelian varieties}, i.e., 
smooth connected non-proper algebraic groups such that every
smooth connected normal affine subgroup is trivial. Moreover,
every pseudo-abelian variety $G$ is commutative. Consider the
$G$-action on itself by multiplication; then the above
associated fiber bundles are exactly the bundles of the form 
$G \times^H H$, where $H \subset G$ is an affine subgroup scheme 
(acting on itself by multiplication) such that $G/H$ is an abelian 
variety. There exists a smallest such subgroup scheme (see 
\cite[9.2 Thm.~1]{BLR}), but no smooth one. 
For a similar example with a projective variety, just 
replace $G$ with a normal projective equivariant completion
(which exists in view of \cite[Thm.~5.2.2]{Brion17}).

To obtain an explicit example, we recall a construction
of pseudo-abelian varieties from \cite[Sec.~6]{Totaro}.
Let $k$ be an imperfect field of characteristic $p$, and
$U$ a smooth connected unipotent group of exponent $p$.
Then there exists an exact sequence of commutative algebraic groups
\[ 0 \longrightarrow \alpha_p \longrightarrow H 
\longrightarrow U \longrightarrow 0, \]
where $H$ contains no non-trivial smooth connected subgroup scheme.
Next, let $A$ be an elliptic curve which is supersingular,
i.e., its Frobenius kernel is $\alpha_p$. Then 
$G := A \times^{\alpha_p} H$ 
is a pseudo-abelian variety, and lies in two exact sequences
\[ 0 \longrightarrow A \longrightarrow G
\longrightarrow U \longrightarrow 0,
\quad  
0 \longrightarrow H \longrightarrow G
\longrightarrow A^{(p)} \longrightarrow 0, \]
since $H/\alpha_p \cong U$ and $A/\alpha_p \cong A^{(p)}$. 

We claim that $H$ is the smallest subgroup scheme $H' \subset G$
such that $G/H'$ is an abelian variety. Indeed, $H' \subset H$ 
and $H'/H$ is finite, hence $\dim(H') = \dim(H) = \dim(U)$. If 
$H' \cap \alpha_p$ is trivial, then the natural map $H' \to U$ 
is an isomorphism. Thus, $H'$ is smooth, a contradiction. 
Hence $H' \supset \alpha_p$, so that the natural map
$H'/\alpha_p \to U$ is an isomorphism; we conclude that $H' = H$. 

In particular, taking for $U$ a $k$-form of the additive group,
we obtain a pseudo-abelian surface $G$. One may easily check 
that $G$ admits a unique normal equivariant completion $X$; moreover, 
the surface $X$ is projective, regular and geometrically integral, 
its boundary $X \setminus G$ is a geometrically irreducible curve, 
homogeneous under the action of $A \subset G$, and 
$X \cong G \times^H Y$, where $Y$ (the schematic closure 
of $H$ in $X$) is not geometrically reduced. Also, the projection 
\[ \psi : X \longrightarrow G/H = A^{(p)} \] 
is the Albanese morphism of $X$, and satisfies 
$\psi_*(\cO_X) = \cO_{A^{(p)}}$. 
\end{remarks}

\subsection{Proof of Corollary \ref{cor:alb}}
\label{subsec:cor}

Recall from \cite{Wittenberg} that there exists an abelian variety
$\Alb^0(X)$,  a torsor $\Alb^1(X)$ under $\Alb^0(X)$, and a morphism 
\[ a_X : X \longrightarrow \Alb^1(X) \] 
satisfying the following universal property: for any morphism 
$f : X \to A^1$, where $A^1$ is a torsor under an abelian variety $A^0$, 
there exists a unique morphism $f^1 : \Alb^1(X) \to A^1$ 
such that $f = f^1 \circ a_X$, and a unique morphism of abelian 
varieties $f^0 : \Alb^0(X) \to A^0$ such that $f^1$ is equivariant 
relative to $f^0$. We then say that $a_X$ is the 
\emph{Albanese morphism}; of course, $\Alb^1(X)$ will be the Albanese 
torsor, and $\Alb^0(X)$ the Albanese variety.

When $X$ is equipped with a $k$-rational point $x$, we may identify 
$\Alb^1(X)$ with $\Alb^0(X)$ by using the $k$-rational point $a_X(x)$
as a base point. This identifies $a_X$ with the universal morphism from 
the pointed variety $(X,x)$ to an abelian variety, which sends $x$ to 
the neutral element.

By the construction in \cite[App.~A]{Wittenberg} via Galois descent, 
the formation of the Albanese morphism commutes with separable
algebraic field extensions. Also, the formation of this morphism
commutes with finite products of pointed, geometrically integral 
varieties (see e.g.~\cite[Cor.~4.1.7]{Brion17}). Using Galois descent 
again, it follows that the formation of the Albanese morphism commutes 
with finite products of arbitrary geometrically integral varieties. 
In view of the functorial considerations of Subsection \ref{subsec:fp}, 
for any such variety $X$ equipped with an action $\alpha$ of 
a smooth connected algebraic group $G$, we obtain a morphism 
of abelian varieties 
\[ \Alb^0(\alpha) : \Alb^0(G) \longrightarrow \Alb^0(X) \] 
such that $a_X$ is equivariant relative to the morphism of algebraic groups
\[ \Alb^0(\alpha) \circ a_G : G \longrightarrow \Alb^0(X). \] 
Also, by  Remark \ref{rem:fin} (i) and \cite[Thm.~4.3.4]{Brion17}, 
the Albanese morphism $a_G : G \to \Alb^0(G)$ can be identified 
with the quotient morphism by the smallest affine subgroup scheme 
$H \subset G$ such that $G/H$ is an abelian variety.

Assume in addition that $X$ is normal and quasi-projective, and $\alpha$
is almost faithful. Then, as proved in Subsection \ref{subsec:model}, 
there exists a $G$-equivariant morphism $\psi : X \to G/H'$, where 
$H' \subset G$ is an affine subgroup scheme such that $G/H'$ 
is an abelian variety; in particular, $H' \supset H$ and $H'/H$ is finite. 
This yields an $\Alb^0(G)$-equivariant morphism of abelian varieties
\[ \psi^0 : \Alb^0(X) \longrightarrow G/H', \]
where $\Alb^0(G) = G/H$ acts on $G/H'$ via the quotient morphism
$G/H \to G/H'$. Since the latter action is almost faithful, so is the action 
of $\Alb^0(G)$ on $\Alb^0(X)$, or equivalently on $\Alb^1(X)$. 

\begin{remark}\label{rem:pr3}
Keep the notation and assumptions of Corollary \ref{cor:bir}, and
assume in addition that $\alpha$ is faithful. Then the kernel
of the induced action $\Alb^1(\alpha)$ (or equivalently, of
$\Alb^0(\alpha)$) can be arbitrarily large, as shown by the following
example from classical projective geometry.

Let $C$ be a smooth projective curve of genus $1$; then $C$ is a torsor
under an elliptic curve $G$. Let $n$ be a positive integer and
consider the $n$th symmetric product $X := C^{(n)}$. This is a smooth 
projective variety of dimension $n$, equipped with a faithful
action of $G$. We may view $X$ as the scheme of effective Cartier
divisors of degree $n$ on $C$; this defines a morphism
\[ f : X \longrightarrow \Pic^n(C), \]
where $\Pic^n(C)$ denotes the Picard scheme of line bundles of
degree $n$ on $C$. The elliptic curve $G$ also acts on $\Pic^n(C)$,
and $f$ is equivariant; moreover, the latter action is transitive 
(over $\bar{k})$ and its kernel is the $n$-torsion subgroup scheme 
$G[n] \subset G$, of order $n^2$. Thus, we may view $\Pic^n(C)$ 
as a torsor under $G/G[n]$. Also, $f$ is a projective bundle, with fiber 
at a line bundle $L$ over a field extension $k'/k$ being the projective 
space $\bP H^0(C_{k'},L)$. It follows that $f$ is the Albanese morphism 
$a_X$. In particular, $\Alb^0(X) \cong G/G[n]$.
\end{remark}

\medskip

\noindent
{\bf Acknowledgements}. Many thanks to St\'ephane Druel and 
Philippe Gille for very helpful discussions on an earlier version of 
this paper, and to Bruno Laurent and an anonymous referee
for valuable comments on the present version. 
Special thanks are due to Olivier Benoist for pointing out that 
Theorem \ref{thm:cover} follows from his prior work \cite{Benoist}, 
and for an important improvement in an earlier version of 
Proposition \ref{prop:ray}.

\end{document}